\documentclass[12pt]{amsart}
\usepackage{graphicx}
\usepackage{amstext}
\usepackage{amsfonts, amssymb}
\usepackage{amsmath}
\usepackage{cancel}
\usepackage{mathtools}
\usepackage{color}
\usepackage{float}
\usepackage{latexsym}
\usepackage{indentfirst}
\usepackage{amssymb}
\usepackage[utf8]{inputenc}
\usepackage{tikz}
\usepackage[utf8]{inputenc}
\usepackage[english]{babel}
\usepackage{amsmath}
\usepackage{amssymb}
\usepackage{amsthm}
\usepackage{hyperref}

\newcommand{\R}{\mathbb{R}}
\newcommand{\Z}{\mathbb{Z}}
\newcommand{\C}{\mathbb{C}}

\newcommand{\K}{\mathcal{K}}

\newcommand{\CP}{\mathbb{C}{\mathrm P}}

\advance\textheight by \topskip
\newtheorem{theorem}{Theorem}

\newtheorem{lemma}[theorem]{Lemma}
\newtheorem{corollary}[theorem]{Corollary}

\theoremstyle{definition}
\newtheorem{remark}{Remark}

\usepackage{fancyhdr}

 \title[Equivariant Minimal sufaces by first eigenfunctions]{Minimal immersions by first eigenfunctions of Riemann surfaces with large conformal automorphism groups. \\The Klein quartic}

\vspace{1cm} 

\author{Antonio Ros\\
 }     

\thanks{This work is supported in part by 
 AEI, Eur. 
 Exc. prog. 
 grant EUR2024-153556,
MINECO/MICINN/FEDER grant PID2020117868GB-I00,
IMAG–Maria de Maeztu grant CEX2020-001105-M / AEI / 10.13039/501100011033, 
MICINN grant PID2020-117868GB-I00 
and regional grant P18-FR4049.}

\begin{document}

\maketitle
\vspace{-.8cm} 
\begin{flushright}
{\small In memory of William H. Meeks III \hspace{.7cm}\textcolor{white}{xxxxx}}
\end{flushright}

\begin{abstract}
Minimal immersions of closed surfaces into spheres by first eigenfunctions play an important role in both the theory of minimal surfaces and spectral geometry. 
In this paper, we establish the existence of equivariant conformal minimal immersions of this type for Hurwitz surfaces and other families of 
surfaces with large automorphism group.

Finally, we explicitly construct a  conformal minimal equivariant immersion
of the Klein quartic into the unit sphere $S^{7}$  by 
first eigenfunctions. Ths immersion is a natural candidate for maximizing the normalized $\lambda_1$ among closed genus-3 surfaces.

\vspace{.1cm}

\noindent

%

\end{abstract}

\tableofcontents


\section{Introduction}


Let $f:(\Sigma, ds^2) \to S^n(1)$ be a minimal immersion of a compact surface into a unit sphere. Then, the coordinate functions of $f$ are eigenfunctions of the Laplacian with eigenvalue $\lambda=2$. Of particular interest is the case when these are first eigenfunctions, that is, $\lambda_1(ds^2)=2$. For genus $\gamma=0$, the round sphere is the only such surface. For $\gamma=1$, there are exactly two possibilities: the equilateral flat torus in $S^5(1)$ and the Clifford torus in $S^3(1)$, \cite{elsoufi}. This property is also shared by the Lawson minimal surfaces in the 3-sphere, as proved by Choe and Soret \cite{choe1}, and by several other interesting examples.

An important problem in spectral geometry is the study of extremal metrics for the first eigenvalue of the Laplacian. In this paper, we focus on this problem within a fixed conformal class. 

Given a compact Riemann surface $\Sigma$, we will consider conformal metrics on $\Sigma$ with  a finite number of conical singularities.
A conformal metric $ds^2$ with a  {\it conical singularity} at $p\in \Sigma$, is expressed as \(ds^2 = e^{2\varphi} \vert{}z\vert{}^\alpha \vert{}dz\vert{}^2\), 
where $z$, \(\vert{}z\vert{} < \varepsilon\), is a complex coordinate  around the point $p$,  \(\varphi\) is a regular function and \(\alpha \) is a real number greater than \(-2\), \(\alpha > -2\). The bound on $\alpha$ ensures that the metric  has a finite local area and diameter, and the total {\it cone angle}  
around the singularity 
\(z = 0\) is given by \(\theta = \pi(\alpha + 2)\). In particular, if \(\alpha = 0\), then the metric is unbranched  and the angle is equal to \(2\pi\).

\vspace{.1cm}

Any  harmonic map  on the Riemann surface into the unit sphere, \mbox{$f\hspace{-0.07cm}:\hspace{-0.03cm}\Sigma \longrightarrow S^{n}(1)$}, 
satifies, using any auxiliary conformal metric without singularities $ds^2$,  the equation \mbox{$\Delta f + |\nabla f|^2 f=0$},
 where the Laplacian and the length of the gradient are taken with respect to \(ds^{2}\).
  If $f$ is non constant, the associated metric 
$
ds^2_f=\frac{1}{2} |\nabla f|^2 ds^2
$ 
 has finitely many conical points and the linear coordinates of $f$ are eigenfunctions of $ds^2_f$ with eigenvalue igual to $2$, see \cite{knpp,helein}. 
 
 \vspace{.1cm}
 
 The {\it normalized first eigenvalue}  of the metric 
 is given by the product of the first eigenvalue of the Laplacian and the area of the metric, $\bar{\lambda}_1(ds^2)=\lambda_1(ds^2)Area(ds^2)$.
Petrides \cite{petrides} shows that, on the space of conformal metrics, this operator  
attains its maximum,
 \[
\Lambda_1(\Sigma) = \sup \left\{ \bar{\lambda}_1(dt^2)
\, \big{|}\,  dt^2 \,\, {\rm is\,\, a \,\, conformal \,\,metric\,\,on \,\,\Sigma }  \right\},
\] 
  and associated with this extremal metric $ds^{2}$, there exists a harmonic map 
  \mbox{$f\hspace{-0.07cm}:\hspace{-0.03cm}\Sigma \longrightarrow S^{n}(1)$}
   into the unit sphere such that $ds^2=ds^2_f$ whose coordinates are first eigenfunctions of the metric.
    
\vspace{.2cm}

The value $\Lambda_1(\Sigma)$ is bounded above in terms of the topology of the surface; see Yang and Yau \cite{yy}. Let $\mathcal{M}\mathrm{etrics}(\gamma)$ and $\mathcal{M}_\gamma$ be the space of Riemannian metrics on a compact surface of genus $\gamma\geq 0$ and the moduli space of compact Riemann surfaces with this topology, 
respectively. The supremum of the normalized first eigenvalue over all metrics on a surface of genus $\gamma$ is denoted by
\[
\Lambda_1(\gamma) = \sup \left\{\bar{\lambda}_1(dt^2)\ \middle| \ dt^2 \in  \mathcal{M}\mathrm{etrics}(\gamma) \right\} =
\sup \left\{ \Lambda_1(\Sigma) \;\middle|\; \Sigma \in \mathcal{M}_\gamma \right\}. 
\]
 For every compact orientable surface of genus $\gamma$, there exists a metric with singularities  
that attains  the maximum  value $\Lambda_1(\gamma)$, see   Petrides, Karpukhin and Stern \cite{petrides2,kps}.
In this case, the extremal metric $ds^2$ determines a Riemann surface $\Sigma\in \mathcal{M}_\gamma$, and the associated harmonic map $f$  
 is conformal. Hence,  $f: (\Sigma, ds^2)\longrightarrow S^n(1)$ is a branched minimal immersion by first eigenfunctions,  
Nadirashvili \cite{nadir}, and the induced metric and the associated  one  coincide, $ds^2=\langle df, df \rangle = ds^2_f$.

    \vspace{.1cm}
    
Maximizing metrics for a given genus are known in only a few cases:
$\Lambda_1(0)=8\pi$, Hersch \cite{hersch}, with the equality holds just for the constant curvature metrics;
Nadirashvili \cite{nadir} proved that  $\Lambda_1(1)=\frac{8}{\sqrt{3}}\pi \approx 14.510\pi$, where the equality holds for the flat equilateral metric,
and Nayatani and Shoda \cite{nayatani} show $\Lambda_1(2)=16\pi$, the equality holding for the Bolza spherical surface and some other close branched spherical metrics.

  \vspace{.2cm}
  
Karpukhin, Kusner, McGrath, and Stern \cite{kusner} investigate metrics that are extremal for the functional $\bar{\lambda}_1$ among those invariant under prescribed symmetry groups and, as a consequence of their analysis, construct several new embedded equivariant compact minimal surfaces in $S^3(1)$.
 Further results on equivariant eigenvalue optimization, including existence and regularity results for maximizing metrics within a conformal class, have been obtained by Vinokurov \cite{vino}.

  \vspace{.1cm}

In this paper, we consider closed Riemann surfaces of genus $\gamma \geq 2$ 
with a large group of holomorphic automorphisms $\operatorname{Aut}(\Sigma)$. 
The most symmetric cases are \textit{Hurwitz surfaces}, whose automorphism groups 
have order $84(\gamma-1)$. The next largest families exhibit orders of 
$48(\gamma-1)$ and $40(\gamma-1)$,  \cite{farkas}.
We study the metrics that maximize the first eigenvalue within these conformal classes and, drawing on the classical theories of Riemann surfaces 
and minimal surfaces and \cite{kusner,vino}, we obtain the following result:
\begin{theorem}
\label{teorema1}
Let $\Sigma$  a compact Riemann surface of genus $\gamma\geq 3$ and $G=\operatorname{Aut}(\Sigma)$ be its group of holomorphic automorphisms. If $|G|=84(\gamma-1)$, $48(\gamma-1)$ or $40(\gamma-1)$, then there exists a $G$-equivariant, unbranched conformal minimal immersion 
$f:\Sigma\longrightarrow S^n(1)\subset\mathbb{R}^{n+1}$ by first eigenfunctions. Furthermore, if \hspace{.01cm} $|G|=84(\gamma-1)$ or $48(\gamma-1)$, then 
the fourth-order Hopf differential $\Omega_2(f)=\langle f_{zz},f_{zz}\rangle dz^4$ vanishes identically. 
\end{theorem}

In genus $3$, the most symmetric Riemann surface is the Klein quartic $\mathcal{K}$, whose automorphism group $\operatorname{Aut}(\mathcal{K})$ is a simple group of order $168$.
This surface is naturally conjectured  to admit a conformal metric that maximizes the normalized  first eigenvalue among all metrics on a surface of genus $3$.
As a first step in this direction,  we describe the corresponding conformal minimal immersion by first eigenfunctions.

\begin{theorem} 
\label{thm:mi_teorema}
There is  a full, unbranched, conformal  minimal immersion 
$f \colon 
{\mathcal K}
\longrightarrow S^7(1)\subset \R^8$ of the Klein quartic $\mathcal K$ into  the unit sphere  by first eigenfunctions.

 \begin{itemize}
 \item[\small{ i)}]  The immersion $f$ 
 is $\rho_6$-equivariant, where $\rho_6$ is the 8-dimensional irreducible representation of
the automorphism group $G_{168}, \mathrm{Table} \hspace{.04cm} \ref{tabla168}$.

\vspace{.05cm}

 \item[\small{ ii)}]  {\it The Hopf differential $\Omega_2(f)$ vanishes.}

   \vspace{.05cm}

  \item[\small{ iii)}] $10.712\pi < \mathrm{Area}(f) < 10.834\pi$.

   \vspace{.05cm}

  \item[\small{ iv)}] The induced metric $ds^2$ maximizes  $\bar{\lambda}_1$ within the conformal class 
 and it is the unique conformal metric on $\mathcal K$ with this property. 
 The multiplicity of the first eigenvalue is $m_1(ds^2)=8$.
 
    \vspace{.05cm}

  \item[\small{ v)}]  
Up to orthogonal congruence, \(f\) is the unique full \(G_{168}\)-equivariant conformal minimal
 immersion of the Klein quartic into a sphere \(S^n(1)\) by first eigenfunctions.
%
   \end{itemize}

\end{theorem}

\section{Automorphism groups}


A closed Riemann surface $\Sigma$ of genus $\gamma\geq 2$ has a finite group of holomorphic automorphisms $\operatorname{Aut}(\Sigma)$. Let $G\leq \operatorname{Aut}(\Sigma)$ be a subgroup of order $|G|$. The quotient $\Sigma/G$ is a compact Riemann surface, and the natural projection
 $\Pi:\Sigma\longrightarrow \Sigma/G$ is a holomorphic map of degree $|G|$.  
 A classical result  states that the order of the group satisfies $|G|\leq 84(\gamma-1)$. More especifically, the possible values are 
 $|G|= 84(\gamma-1)$, $ 48(\gamma-1)$, $40(\gamma-1)$ or less, (see Farkas, Kra \cite{farkas} pp. 258--260).
  The surfaces that attaint the maximum value for $|G|$ are  the {\it Hurwitz surfaces}.  
  
\vspace{.1cm}

 For each point $p \in \Sigma$, its stabilizer $G_p= \{\phi \in G \, |\,  \phi(p)=p \}$ is a finite 
cyclic subgroup. The set of points in $\Sigma$  with non-trivial stabilizers is finite and the order of each $G_p$ coincides with the ramification index of the branch point $\Pi(q) \in \Sigma/G$. 

\vspace{.1cm}

The geometric behavior of the pair $(\Sigma,G)$ 
allows us to define the {\it signature}
\[
(\gamma_0; m_1, \dots, m_r),
\] 
where $g_0$ is the genus of the 
quotient Riemann surface $\Sigma/G$, and the integers $2\leq m_1\leq \cdots \leq m_r$ 
represent the unique orders of the non-trivial cyclic stabilizers associated with each of the $r$ branch points of $\Pi$.
Moreover, the Riemann-Hurwitz formula applied the projection map, relates the genus of $\Sigma$, the order of $G$  and  the signature,
\begin{equation}
\label{RH}
g-1 = |G|(g_0-1) + \frac{|G|}{2} 
\sum_{i=1}^{r} \big(1- \frac{1}{m_i}\big).
\end{equation}

The pairs $(\Sigma, G)$ with $\vert{}G\vert{} > 4(g-1)$ are called  {\it Riemann surfaces with large automorphism group}, see Kukarni \cite{kulkarni}. 
Under this condition, it is known that the genus of the quotient surface is $g_0=0$ and $r=3$ or $4$. 

If $r=3$, the signature is said to be {\it triangular}, and the three most symmetric cases above correspond to the signatures $(0;2,3,7)$, $(0;2,3,8)$ and $(0;2,4,5)$, respectively.
 Magaard, Shaska, Shpectorov, and Völklein \cite{shaska2} provide tables of compact Riemann surfaces of genus $\gamma\leq 10$ with complete automorphism groups $G=\text{Aut}(\Sigma)$ satisfying $\vert{}G\vert{}>4(\gamma-1)$. 
 
 \vspace{.2cm}

For any $p\in \Sigma $, the orbit $G \cdot p = \{\phi(p)\ | \ \phi\in G\}$, 
is related to the stabilizer $G_p$ by the Orbit-Stabilizer Theorem $|G \cdot p|\cdot |G_p| = |G|$ and, as the
coefficients $m_i$  are ordered, the size of the smallest 
  orbit of $G$ is given as
\[
\mathfrak{o}(G) = \min \{|G \cdot p| \  \, \big/ \,  p\in \Sigma, |G_p| > 1 \} = \frac{|G|}{m_r}.
\]
In this paper, we focus on the case $\mathfrak{o}(G) > 4(g-1)$. By analogy with the previous case, we say that the pair $(\Sigma, G)$ is a Riemann surface with  $\textit{large\ orbits}$. Note that this condition is stronger than that of a large automorphism group.

 \vspace{.1cm}

As a novel result of this section, we classify automorphism groups with large orbits, 
building on Breuer's comprehensive study of large automorphism groups (\cite{breuer}, Chapter 1, p. 16).
\begin{theorem}
\label{o(G)}
Let $\Sigma$ be a closed Riemann surface of genus $\gamma \geq 2$, 
$G$ an automorphism group of $\Sigma$, and 
let $\mathfrak{o}(G)$ denote the cardinality of the smallest orbit of $G$. 
Then, the condition $\mathfrak{o}(G) > 4(\gamma-1)$ occurs just in the following cases:
\begin{equation}
\renewcommand{\arraystretch}{1.2}
\begin{tabular}{ccc}
Signature &  $\vert{}G\vert{}$  & $\mathfrak{o}(G)$ \\
\hline
$(0; 2,3,7)$ & $84(\gamma-1)$ & $12(\gamma-1)$ \\
$(0; 2,3,8)$ & $48(\gamma-1)$ & $6(\gamma-1)$ \\
$(0; 2,4,5)$ & $40(\gamma-1)$ & $8(\gamma-1)$  \\
$(0; 3,3,4)$ & $24(\gamma-1)$ & $6(\gamma-1)$
\end{tabular}
\label{tabla1}
\end{equation}
\end{theorem}
\begin{proof}
The hypothesis  $4(\gamma-1) <  \mathfrak{o}(G)$ implies that   $4(\gamma-1) <  |G|$, i.e. $G$ { is a large group of automorphisms of} $\Sigma$. 
Then, from  Breuer \cite{breuer}, Lemma 3.18, it  follows that the signature verifies  $\gamma_0=0$ and $r=3$ or $4$.

\vspace{.1cm}

{\bf Case (i)} {$ r=4$.} 

We show that this case is impossible. The signature is $(0; m_1,m_2,m_3,m_4)$, the smallest orbit satisfies $\mathfrak{o}(G)\hspace{.02cm}m_4= |G|$ and, from
Lemma 3.18, \cite{breuer}, $|G|\leq 12(\gamma-1)$. Therefore

\[
m_4 <\frac{ \mathfrak{o}(G)\hspace{.02cm}m_4}{4(\gamma-1)} = \frac{|G|}{4(\gamma-1)} \leq 3.
\]
Hence, $m_4=2$ and the signature is  $(0; 2,2,2,2)$. Then, Riemann-Hurwitz formula gives $\gamma=1$, a contradiction. 

\vspace{.2cm}

{\bf Case (ii)} {\it $r=3$ and $|G| \geq 24(\gamma-1)$.}

 In this case, Breuer \cite{breuer}, p. 16,  gives a complete list of the corresponding triangular signatures and  the orders of $G$,
\[
\begin{array}{c|c} \text{Signature} & |G| \\
\hline 
\hspace{.01cm}
 (0;2,3,7) & 84(\gamma-1)\hspace{.06cm} \\
 (0;2,3,8) & 48(\gamma-1) \\
   (0;2,4,5) & 40(\gamma-1) \\ 
 (0;2,3,9) & 36(\gamma-1) \\
 (0;2,3,10) &  30 (\gamma-1) \\
  \end{array}
  \hspace{1cm}
  \begin{array}{c|c} \text{Signature} & |G| \\
\hline 
\hspace{.01cm}
 (0;2,3,11)& \hspace{.01cm}(132/5) (\gamma-1)\\
  (0;2,3,12)& 24(\gamma-1)\\ 
   (0;2,4,6)& 24(\gamma-1)\\ 
    (0;3,3,4)& 24(\gamma-1) \\ 
   \textcolor{white}{E} \\
  \end{array}
\]
Since $\mathfrak{o}(G) = \vert{}G\vert{}/m_3$, we have that the only cases that satisfy the hypotheses of  the theorem are $(0;2,3,7)$, $(0;2,3,8)$, 
$(0;2,4,5)$, and $(0;3,3,4)$.

\vspace{.2cm}

------------

\vspace{.2cm}

{\bf Case (iii)} {\it $r=3$  and $|G| <24(\gamma-1)$.}  

The signature is $(0; m_1,m_2,m_3)$, with $2\leq m_1\leq m_2 \leq m_3$, and  the Riemann-Hurwitz formula  (\ref{RH} ) states that
\begin{equation}
\label{rh1}
\frac{1}{m_1}+\frac{1}{m_2}+\frac{1}{m_3}<1 \hspace{.3cm}  {\rm and }  \hspace{.3cm}
|G|=
\frac{2(\gamma-1)}{1-\frac{1}{m_1}-\frac{1}{m_2}-\frac{1}{m_3}}.
\end{equation}
First we establish the following inequalities:

\vspace{-.5cm}

\begin{gather}
\label{222}
\frac{1}{m_1}+\frac{1}{m_2}+\frac{1}{m_3} < \frac{11}{12}, \\[1.5ex]
\label{333}
 \frac{1}{m_1} + \frac{1}{m_2}+ \frac{3}{2m_3} >1.
\end{gather}

\vspace{-.1cm}

Indeed, the assumption $\vert{}G\vert{} < 24(\gamma-1)$  together with (\ref{rh1}), give 
\vspace{-.1cm}
\[
\frac{2}
{1-\frac{1}{m_1}-\frac{1}{m_2}-\frac{1}{m_3}} = \frac{|G|}{\gamma-1} < 24,
\]
which transforms into (\ref{222}).

Similarly, combining (\ref{rh1}) with the assumption on $\mathfrak{o}(G)$, 
\vspace{-.1cm}
\[
4 m_3 < \frac{\mathfrak{o}(G)m_3}{\gamma-1} = \frac{|G|}{\gamma-1} = \frac{2}{1-\frac{1}{m_1}-\frac{1}{m_2}-\frac{1}{m_3}},
\]
\vspace{-.1cm}
and, after rearrangement, we obtain (\ref{333}).   

\vspace{.2cm} 

Finally, subtracting \eqref{222} from \eqref{333}, we obtain $1/(2m_3) > 1/12$, which yields $m_3 \leq 5$, and 
a direct check shows that no triplet of integers $2 \le m_1 \le m_2 \le m_3 \le 5$ simultaneously satisfies both \eqref{222} and \eqref{333}.
Therefore, case (iii) is impossible and this concludes the proof of the theorem.
\end{proof}

\vspace{.1cm}

Hurwitz surfaces are  those with signature $(2,3,7)$.  It is known that each of the four 
cases in the family (\ref{tabla1}) occurs for infinitely many values of $\gamma$ and is impossible for infinitely many others. Futhermore, the genus and order of the group do not uniquely determine the Riemann surface; for instance, there are three distinct Hurwitz surfaces of genus $\gamma=14$.

\vspace{.2cm}

The following table lists the compact Riemann surfaces of the family $\mathfrak{o}(G)>4(\gamma-1)$
with the smallest genera, as provided by  Magaard, Shaska,  Shpectorov and  V\"olklein \cite{shaska2}:

\begin{table}[H]

\centering

\begin{tabular}{clccccll}


$\gamma$ &  Riemann surface & Curve equation  & $(p,q,r)$ & Group ID & References\\ 

 \hline\vspace{-.3cm}\\ 

2 & Bolza\ curve  & $y^2=x(x^4-1)$  & $(2,3,8)$ & $(48,29)$ & \cite{nayatani} \vspace{.1cm}\\

3 & Klein\ quartic\ & $x^3 y + y^3 z + z^3 x=0$  & $(2,3,7)$ & $(168,42)$  & \cite{karcher}  \vspace{.1cm}\\

3 & Fermat\ quartic & $x^4+ y^4+z^4=0$ &  $(2,3,8)$ & $(96,64)$ & \cite{karcher}

\vspace{.1cm}\\

4 & Bring curve &  


\ in  $\CP^3$  

& $(2,4,5)$ & $(120,34)$ & \cite{bring}   

\\

2 & Bolza\ curve & & $(3,3,4)$ & $(24,3) $ & \cite{bujalancecc}  

\\

3 & Fernat quartic &  & $(3,3,4)$  &  \ $(48,3)$  & \cite{bujalancecc}  

\end{tabular}

\centering
\vspace{.2cm}
\caption{Compact Riemann surfaces  with genus $\gamma\leq 4$ satisfying  $\mathfrak{o}(G) >4(\gamma-1)$.
\label{family3} Bolza curve is the unique hyperelliptic surface in the family. The next 
ones are $G$-equivariant unbranched curves in the complex projective space. 
Genus three  surfaces are 
quartic curves in  $\CP^2$ and   
the Bring curve of genus $4$ in  $\CP^3$  is the complete intersection of a quadric and a cubic surface.
The automorphism group is given according to the
GAP Small Groups Library, \cite{tim,gap,shaska2}. 
The first number is the order of the group and the second one represent the specific position of that group 
within the catalog of all non-isomorphic groups of that same order.
For the last two  surfaces, see Remark \ref{bu}. 
}
\end{table}
\begin{remark}
\label{bu}
Any surface $(\Sigma,G)$ of genus $\gamma\geq 2$ with signature $(0;2,3,8)$ has a unique subgroup $H\leq G$ of index $2$, and the signature of $(\Sigma,H)$ is $(0;3,3,4)$ (see Bujalance, Cirre and Conder~\cite{bujalancecc}). However, not every surface with signature $(0;3,3,4)$ arises in this way. For instance, for $\gamma=10$, there is exactly one topological group action with signature $(0;2,3,8)$ but two distinct cases with signature $(0;3,3,4)$ (see Broughton et al. \cite{broughton1}, p. 51). 
\end{remark}
\begin{remark}
Although many, but not all, surfaces with $\mathfrak{o}(G) >4(\gamma-1)$ admit antiholomorphic transformations,
 in this paper we focuses exclusively on holomorphic ones.  
For the antiholomorphic case see Singerman \cite{singerman}.
\end{remark}

\vspace{.1cm}

We will also need to control the transformation group of surfaces with boundary.
If ${\Sigma}$ is a compact Riemann surface with a nonempty boundary, the order of its automorphism group can also be bounded. 
According to Oikawa \cite{oikawa}, if  
 $\gamma\geq 0$ and $r\geq 1$ denote  the genus and the number of boundary components of $\Sigma$, the Euler characteristic is 
 $\chi(\Sigma)=2-2\gamma-r$.
 Except for the disk $(\gamma,r)=(0,1)$ and the annulus $ (0,2)$, the group of holomorphic automorphisms is finite and satisfies:
\begin{equation}
\label{borde}
|Aut(\Sigma)|\leq \min \big\{12(\gamma-1) + 6r,\  84(\gamma-1)\big\}.
\end{equation}

\vspace{.2cm}

\section{First eigenfunctions of the Laplacian }

Let  $\Sigma$ be a closed Riemann surface. In this section, we consider $(\Sigma,ds^2)$
 to be an unbranched conformal Riemannian metric. For any holomorphic coordinate $(U,z)$, there is a smooth function 
 $\varphi$ on $U\subset \C$ such that $ds^2= e^{2\varphi}|dz|^2$. 

The subspace of first eigenfunctions of the Laplacian is given by
\[
V_1(ds^2) = \big\{ u\in C^\infty(\Sigma) \ \big| \ \Delta u + \lambda_1 u=0 \big\}.
\]
The {\it nodal} $1$-cycle of a nonzero eigenfunction $u$ is  $N(u)=\{p\in \Sigma \ / \ u(p)=0\}$.

For any  $p \in N(u)$, in holomorphic coordinate $z=x+iy$ centered at $p$, 
\[
u(x,y)=P_m(x,y)+ O(r^{m+1}),  \hspace{.4cm}  r=\sqrt{x^2+y^2}<\varepsilon,
\]
where $P_m$ is a non-zero harmonic polynomial of degree $m \geq 1$. 

The  {\it order  of vanishing}  of $u$ at the point $p$ is $m$ 
and the nodal cycle has a finite number of {\it vertices}, i.e.,
points $p_1,\ldots, p_V \in N(u)$, 
 with order of vanishing  $m_j \geq 2 $ for $j=1,\ldots, V$. 

Around a vertex $p_j$ the nodal $1$-cycle consists of a $C^1$ equiangular system of $m_j$ curves.

Finally, the Courant nodal domain theorem implies that the {\it nodal domains} $D_1=\{ u> 0\}$ and $D_2 = \{ u < 0\}$ are open connected  subsets of $\Sigma$ with piecewise smooth boundaries, \cite{cheng}.

The following result relates the topology of the surfaces to the vertices of the nodal cycle $N(u)$; Cheng \cite{cheng}, Besson \cite{besson}, see also \cite{choe}, \cite{vergasta}.

\begin{lemma}

\label{choe}

Let $ds^2$ be a conformal metric on a compact Riemann surface $\Sigma$ of genus $\gamma \geq 0$ and $u:\Sigma\longrightarrow \R$ be a nonzero first eigenfunction of the Laplacian. 
Then, the orders of the zeros of $u$ at the vertices satisfy  
\[
\sum_{j=1}^V (m_j -1) \leq 2\gamma,
\]
and equality holds if and only if the two nodal domains are disks.
\end{lemma}

\begin{proof}
Using the Gauss-Bonnet theorem for a domain with  piecewise smooth boundary we obtain 
\[
\int_{D_i} K dA + \int_{\partial D_i} \kappa_g(s)ds  + \sum_{j} \theta_j = 2\pi \chi(D_i), 
\]
where $K$ is Gaussian curvature, $\kappa_g$ the geodesic curvature along  $\partial D_i$, and $\sum_j\theta_j$ is the sum of all the exterior angles at the vertices of the nodal domain $D_i$, $i=1,2$. 

\vspace{.2cm}

The proof  of the lemma is obtained 
by adding the above equation for $i=1,2$  and taking into account the following points:
\begin{itemize}
\item[$\cdot$] The geodesic curvature takes opposite values on the two sides of the nodal 1-cycle  $N(u)$ and, thus,  
its integral cancels out.
\item[$\cdot$]  The sum of exterior angles at a vertex $p_j$ of $N(u)$ is equal to $2\pi \big( m_j -1\big)$. 
\item[$\cdot$]  The nodal domains $D_i$ are connected and, so, $\chi(D_i)\leq 1$, $i=1,2$.
\item[$\cdot$]  The Gauss-Bonnet, again, gives  $\int_{\Sigma} K=2\pi(2-2g)$.
\end{itemize}
Therefore
\[
2\pi(2-2\gamma) + 2\pi \sum_j (m_j-1)  =\int_\Sigma K \ dA + 2\pi \sum_j (m_j-1) = 2\pi \big( \chi(D_1) +\chi(D_2) \big) \leq 4\pi.
\]
\end{proof}

Let $m_1(\gamma)$ be the maximum of the  dimensions $m_1(ds^2)=\dim V_1(ds^2)$ of the first eigenspace of the Laplacian among all Riemannian metrics on a compact orientable surfaces of genus $\gamma$.
Using Lemma \ref{choe}, the sharp bounds $m_1(0)=3$ and $m_1(1)=6$ 
are obtained in \cite{cheng} and \cite{besson}, respectively.

\vspace{.1cm}

For $\gamma\geq 2$, the best known estimate uses a different argument, Sevennec \cite{sevennec},
\begin{equation}
\label{sevennec1}
m_1(\gamma) \leq 2\gamma+3.
\end{equation}

\vspace{.1cm}

\subsection{Group Representations on the First  Eigenspace} \textcolor{white}{.}

\vspace{.1cm}

If the metric $ds^2$ is invariant under the finite group of holomorphic transformations on the compact Riemann surface
 $\Sigma$, $G \leq \text{Aut}(\Sigma)$,  then the first eigenspace $V_1(ds^2)$ is preserved by the action of $G$. 
 
For any irreducible orthogonal representation
  $\rho: G \longrightarrow O(m)$ of dimension $\dim \hspace{.005cm} \rho=m$,  we denote by $V_1(ds^2,\rho)$ the corresponding 
component of the first eigenspace, 
\[
V_1(ds^2,\rho) = \{u \in V_1(ds^2) \, \big| \, u \circ \phi^{-1} = \rho(\phi)u, \;\forall \phi \in G\},
\]
whose  dimension is $i \times m$, for some $i=0,1,2,\ldots$

\vspace{.1cm}

If  $\{\rho_1, \dots, \rho_k\}$ is the finite set of irreducible orthogonal representations of $G$, then the space of the first eigenfunctions decomposes into the direct sum
\[
V_{1}(ds^2)= V_1(ds^2,\rho_1) \oplus \cdots \oplus V_1(ds^2,\rho_k).
\]
The {\it trivial representation} $\rho_1: G \longrightarrow O(1)$ corresponds to the trivial homomorphism $\rho_1(\phi) = 1$, 
$\forall \phi \in G$, and the subspace $V_1(ds^2,\rho_1)$ is given by the invariant first eigenfunctions,  $u\circ\phi=u$ for all isometries of $G$.
As an illustrative example, if the group $G=\{1, \phi\}$ is given by an involutive isometry $\phi$, 
then we have two irreducible representations $\{\rho_1,\rho_2\}$, and the direct sum $V_1(ds^2)=V_1(ds^2,\rho_1)\oplus V_1(ds^2,\rho_2)$ corresponds to the even and odd eigenfunctions, respectively.

\begin{theorem}
\label{cota}
Let $(\Sigma,G)$ be  a compact Riemann surface of genus $\gamma\geq 2$ with  a  group of holomorphic automorphisms $G$ and let $ds^2$ be  a $G$-invariant conformal metric on $\Sigma$. 
\mbox{If \hspace{.01cm} $|G|>12g$ and  $\mathfrak{o}(G)>2\gamma$, } 
then the invariant subspace  
corresponding to trivial representation $\rho_1$ of $G$ satisfies $V_1(ds^2, \rho_1)=0$. 
\end{theorem}
\begin{proof}
 Suppose that there is nonzero eigenfunction $u \in V_1(ds^2, \rho_1)$. 
Then $u$  is invariant for all transformations of $G$ and 
the nodal cycle $N(u) = \{ p \in \Sigma / u(p)=0 \}$ is the union of a number of immersed curves intersecting in a equiangular way at its vertices.

\vspace{.1cm}

 $N(u)$ divides $\Sigma$ in two connected nodal domains $D_1 = \{u>0\}$ and \mbox{$D_2=\{ u<0\}$} with  piecewise smooth boundary.
If $N(u)$ has any singular point then the set of vertices of $N(u)$ contains an orbit of $G$ and lemma \ref{choe} gives
  \[
 \mathfrak{o}(G) \leq \sum_k (m_k-1) \leq 2\gamma,
 \]
 a contradiction. Therefore $N(u)$ has no vertices.

  \vspace{.2cm}

  Since the nodal cycle has no vertices, $N(u)$ is a disjoint union of $r\geq 1$ embedded smooth closed curves. Then 
  $\chi(\Sigma)= \chi(D_1)+ \chi(D_2)$ and,  
  thus, we have 
$
2-2\gamma =  (2 - 2\gamma_1 -r)+(2-2\gamma_2-r) 
$,
which simplifies to
$
\gamma_1+\gamma_2 +r= \gamma+1.
$

\vspace{.2cm}

Now we consider the two following cases:

\vspace{.1cm}

(1) {\it The two nodal domains verify the assumption of  Oikawa result (\ref{borde}).}

 Then 
\mbox{$
|G| \leq |Aut(D_i) |\leq $} $12(\gamma_i-1)+6r 
$, $i=1,2$. 
Summing in $i$ we get
\[
2|G| \leq 12(\gamma_1+\gamma_2) +12r-24 = 12(\gamma+1)-24 = 12(\gamma-1),
\]
which is impossible as we are supposing that $|G|>12\gamma$.

\vspace{.2cm}

(2) {\it One of the domains, say $D_2$, does not verify this assumption.}

Then $r\leq 2$, $D_1$ verifies (\ref{borde}) and, using that $\gamma_1\leq \gamma$, we obtain 
\[
12 \gamma < |G| \leq|Aut (D_1 )|
\leq 12(\gamma_1-1)+6\times 2 \leq  12\gamma,
\]
a contradiccion. 
\end{proof}

As a direct consequence of Theorems \ref{o(G)} and \ref{cota}, we obtain the following result.

\begin{corollary}
\label{chi1} 
Let $\Sigma$ be a closed Riemann surface of genus $\gamma \geq 3$ with an automorphism group $G$ with 
triangular signature $(0; 2,3,7)$, $(0; 2,3,8)$, $(0; 2,4,5)$ or $(0; 3,3,4)$. Then, for any $G$-invariant conformal metric $ds^2$, the subspace of $G$-invariant first eigenfunctions is $V_1(ds^2,\rho_1)=\{0\}$.
\end{corollary}

\section{Equivariant spherical harmonic maps and  minimal surfaces } 
\label{harmonicmaps}

Let $\Sigma$ be a compact Riemann surface and $G$ a finite group of holomorphic automorphisms of $\Sigma$. 
A \textit{ (real) orthogonal representation} is a group homomorphism \mbox{$\rho: G \longrightarrow O(n+1)$}  from $G$ to the orthogonal group.
This representation naturally acts on the Euclidean space $\mathbb{R}^{n+1}$. 
We say that a map $f: \Sigma \longrightarrow \mathbb{R}^{n+1}$ is $\rho$-\textit{equivariant} if
\[
f(\phi \cdot p) = \rho(\phi) f(p) \quad \forall \phi \in G \text{ and } p \in \Sigma,
\] 
and that $f$ is $G$-\textit{equivariant} if it is $\rho$\hspace{.1cm}-equivariant for some orthogonal representation $\rho$.

\vspace{.1cm}

The representation  $\rho$ is \textit{irreducible} 
if the only $\rho$\hspace{.1cm}-invariant vector subspaces of $\mathbb{R}^{n+1}$ are the trivial ones; that is, if a subspace $V \subset \mathbb{R}^{n+1}$ satisfies 
\[
\rho(\phi)V \hspace{-.04cm}\subset V,  \ \forall \phi \in G, 
\ \ {\rm then} \ \ V = \{0\} \  {\rm or} \  V = \mathbb{R}^{n+1}. 
\]

\vspace{.2cm}

\subsection{Harmonic maps.}   \textcolor{white}{.}

\vspace{.1cm}

The energy of a differentiable map $f$ from $\Sigma$ to a sphere is given by  
\[
E(f) =  \frac{1}{2}\int_\Sigma |\nabla f|^2 \, dA,
\]
where the gradient, $|\nabla f|^2$ and $dA$, are defined with respect to 
an arbitrary conformal metric $ds^2$. 

A {\it harmonic map} $f: \Sigma \longrightarrow S^n(1) \subset \mathbb{R}^{n+1}$ is a smooth map that, in terms of 
$ds^{2}$, satisfies the vector equation
$
\Delta f + |\nabla f|^2 f = 0
$.
Any weakly differentiable map that is a critical point of the energy functional is a smooth harmonic map,  Hélein \cite{helein}.
 

\vspace{.2cm}

For  any  nonconstant  harmonic map $f:\Sigma\longrightarrow S^n(1)\subset \R^{n+1}$,  the metric  
 \[
 ds^2_f=  {\frac{1}{2}}
 |\nabla f|^2   ds^2,
\]
is a conformal  metric with a finite number of conical singularities and the lineal functions of $f$ are eigenfunctions of $ds^2_f$ of eigenvalue $\lambda =2$,
 Karpukhin, Nadirashvili, Penskoi and Polterovich \cite{knpp}.

\vspace{.2cm}

The {\it normalized first eigenvalue} of $ds^2$ is defined as $\bar{\lambda}_1(ds^2)= \lambda_1(ds^2)Area(ds^2)$,
and if  we consider the supremum 
\begin{equation}
\label{conformal}
\Lambda_1(\Sigma) = \sup \big\{ \bar{\lambda}_1(dt^2) \ \big| \ dt^2 \text{ is a conformal metric on } \Sigma \big\},
\end{equation}
then there exists a conformal metric $ds^2$ with conical singularities that maximizes the functional, 
$\bar{\lambda}_1(ds^2) = \Lambda_1(\Sigma)$, see  Petrides \cite{petrides},  \cite{knpp} and Karpukhin and Stern \cite{karpukhin-stern}.
This extremal metric gives rise to a harmonic map  $f: \Sigma \longrightarrow S^n(1)\subset \R^{n+1}$, with $ds^2= ds^2_f$, parametrized by  first eigenfunctions of $ds^2_f$.

\vspace{.1cm}

Given a finite group of automorphisms  $G$ on a compact Riemann surface $\Sigma$,   
 for weakly differentiable \mbox{$G$-equivariant} harmonic maps 
$f:\Sigma\longrightarrow S^n(1)$, 
we have the regularity of these maps away from the fixed points of $G$, \cite{helein}, and 
 Sacks and Uhlenbeck \cite{SU} proved that any harmonic map of finite energy on a Riemann surface, which is regular except possibly at isolated points, admits a smooth harmonic extension to the entire surface.

\subsection{{Hopf differentials.}} \textcolor{white}{.}

\vspace{.1cm}

For any harmonic map $f:\Sigma\longrightarrow S^n(1)$, the {\it first Hopf differential} is 
the quadratic differential $\Omega_1(f)=\langle f_z,f_z\rangle dz^2$,  where $z$ is a local holomorphic coordinate. This differential is 
 holomorphic. 

\vspace{.1cm}

If $f$ is nonconstant, 
then  $\Omega_1(f)$ vanishes 
if and only if $f$ is a branched minimal immersion, i.e. a conformal harmonic map. 
 In this case, the quartic {\it Hopf differential} 
$\Omega_2(f)= \langle f_{zz},f_{zz}\rangle dz^4$ is holomorphic. 
If $\Omega_2\equiv 0$, then the next {\it Hopf differential} 
$\Omega_3= \langle f_{zzz},f_{zzz}\rangle dz^6$ is holomorphic and the process  continues with the sucesive
$2k$-diferentials  $\Omega_k$,
 until we arrive (if it happens)  to the first non zero (and necessarily holomorphic) 
 one, \cite{lawson1}, 
  \cite{vlachos}. 

 These objects play an important role in the theory of minimal surfaces of the sphere and a number of relevant results depend to its behavior,
Bolton, Pedit and Woodward \cite{pedit},   Miyaoka \cite{miyaoka},  Bryant \cite{bryant}, Eschenburg and Vlachos
  \cite{eschen-vlachos, vlachos}.

\vspace{.2cm}

For $G$-equivariant harmonic maps with smallest orbit of length $\mathfrak{o}(G)$ we have the following result.

\begin{lemma} 
\label{hopf}
 Let $\Sigma$ a closed Riemann surface of genus $\gamma\geq 2$,
and $f\hspace{-.1cm}:\hspace{-.1cm}\Sigma\longrightarrow S^n(1)$ be a $G$-equivariant harmonic map.
\begin{itemize}
\item[(i)] If 
  $\mathfrak{o}(G)> 4(\gamma-1)$, then $f$ is a branched conformal minimal surface.
 \item[(ii)]  If  
  $\mathfrak{o}(G)>8(\gamma-1)$, then the 
 Hopf differential $\Omega_2=\langle f_{zz},f_{zz}\rangle dz^4 $ is equal to zero.
  \end{itemize}
  \label{hopfminimal}
\end{lemma}
\begin{proof}
 As $f$ is equivariant, holomorphic tranformations in $G$ 
preserve the set of zeroes of  $\Omega_k$. Then  the result follows because  a non zero $2k$-holomorphic differential on $\Sigma$ is annulled  just 
at $4k(\gamma-1)$ points, counted with multiplicity, \cite{farkas}.   
\end{proof}    

\vspace{.15cm}

\subsection{Branched minimal surfaces.}
\vspace{.1cm}
The study of branch points of minimal surfaces is a classical and interesting subject, Gulliver and Osserman  \cite {osserman}, Eschenburg and Tribuzy  \cite{tribuzy}, Micallef and White \cite{micallef--white}.

By adapting  some of these results to our situation, we conclude the following;

\begin{theorem} 
\label{unbranched}
{\it Let $f:\Sigma\longrightarrow S^n(1)$ be a $G$-equivariant branched conformal minimal surface. If
 \begin{equation}
 Area(f) <2\pi \big( \mathfrak{o}(G) - 2(\gamma-1) \big)
 \label{branching}
 \end{equation} 
 then $f$ is unbranched.}
 \end{theorem}
\begin{proof}
 Let $K$ and $\sigma$ be the Gauss curvature and the second fundamental form  of the immersion. The Gauss equation gives
\begin{equation}
|\sigma|^2 = 2-2K.
\label{gauss1}
\end{equation}
The Gauss-Bonnet theorem holds not only for regular metrics but also for branched ones, see \cite{tribuzy, troyanov}, 
the first reference consider branched minimal surfaces and the other general conformal metrics with conical singularities.
The total branching number $b$ of the minimal immersion $f$ is a nonnegative integer and the Gauss-Bonnet theorem gives
\[
\int K dA = 2\pi (2-2\gamma+b).
\]
From the above two equalities, as $|\sigma|\geq 0$ we have  
\begin{equation}
\label{b>0}
b \leq \frac{1}{2\pi}Area(f) +2(\gamma-1).
\end{equation}
If $b>0$,  then $b\geq \mathfrak{o}(G)$ and the result follows.
\end{proof}

\vspace{.2cm}

\section{Minimal inmersions of surfaces by first eigenfunctions}

Let $\Sigma$ be a compact Riemann surface of genus $\gamma$. For any conformal Riemannian metric $ds^2$ on $\Sigma$, we consider the {\it normalized first eigenvalue} of the Laplacian,
\[
\bar{\lambda}_1(ds^2) = \lambda_1(ds^2)\operatorname{area}(ds^2).
\]
This functional is bounded above in terms of the genus by the Yang--Yau inequality \cite{yy},
\[
\bar{\lambda}_1(ds^2) \leq 8\pi \Big[\frac{\gamma+3}{2}\Big],
\]
where $[x]$ denotes the integer part of $x$. For sharper upper bounds, see \cite{kv,ros3}.

For a fixed Riemann surface $\Sigma$, we define
\[
\Lambda_1(\Sigma) =
\sup\big\{\bar{\lambda}_1(dt^2) \mid
dt^2 \text{ is a conformal metric on } \Sigma\big\}.
\]
Petrides \cite{petrides} proved that this supremum is attained by a conformal metric $ds^2$ with finitely many conical singularities. 
Moreover, the maximizing metric gives rise to a harmonic map
\[f:\Sigma \longrightarrow S^{m}(1)\subset \mathbb{R}^{m+1},\]
such that \(ds^2 = ds_f^2\), and whose coordinate functions are first eigenfunctions of \(ds^{2}\).

\vspace{.1cm}

Every branched conformal minimal immersion $f:(\Sigma,ds^2)\longrightarrow S^m(1)$ provides an upper bound for the normalized first eigenvalue over the conformal class of $\Sigma$. Indeed, by the results of Li and Yau on conformal area \cite{liyau},
\begin{equation}
\label{liyau1}
\Lambda_1(\Sigma)\leq 2\operatorname{Area}(f),
\end{equation}
with equality if and only if $f$ is a minimal immersion by first eigenfunctions, that is,
\[
\Delta f+2f=0
\hspace{.4cm}\text{and}\hspace{.4cm}
\lambda_1(ds^2)=2.
\]

For minimal immersions by first eigenfunctions, the following theorem provides a strong uniqueness result for conformal metrics. 
It was established by Montiel and Ros \cite[Theorem 1]{montiel-ros1} in the unbranched case and 
extended to the general setting by Cianci, Karpukhin and Medvedev \cite[Theorem 1.4]{karpu1}.

\begin{theorem}[\cite{montiel-ros1,karpu1}]
\label{mrkarpu}
Let $\Sigma$ be a compact Riemann surface of genus $g\geq 1$. Then there exists at most one conformal metric $ds^2$ on $\Sigma$, possibly with singularities at branch points, that admits a full branched minimal immersion by first eigenfunctions
\[
f:(\Sigma,ds^2)\longrightarrow S^n(1)
\]
for some $n\geq 3$. Moreover, if such an immersion exists, then
\(
\bar{\lambda}_1(ds^2)=\Lambda_1(\Sigma)
\)
 and $ds^2$ is the unique maximizer of the normalized first eigenvalue within its conformal class.
\end{theorem}

Note that the theorem establishes the uniqueness of the maximizing metric $ds^2$, but not necessarily that of the associated minimal immersion $f$. The situation is different when $n=2$, corresponding to holomorphic or antiholomorphic maps from $\Sigma$ to $S^2(1)$, for which uniqueness of the metric need not hold; see \cite{karpu1}.

\vspace{.2cm}

Let \(\mathcal{M}\text{etrics}(\gamma)\) denote the space of all metrics with conical singularities on a compact, 
orientable smooth surface of genus \(\gamma\).
Although the functional \(\bar{\lambda}_1:\mathcal{M}\text{etrics}(\gamma)\longrightarrow \mathbb{R}\) is not differentiable, 
it is only Lipschitz continuous, minimal immersions by first eigenfunctions play the role of critical points for this functional, as established by Berger \cite{berger}, Nadirashvili \cite{nadir}, and El Soufi and Ilias \cite{elsoufi}.

\vspace{.2cm}

A central problem in geometric spectral theory is to maximize the first eigenvalue of the Laplacian over surfaces of fixed genus. We denote by 
$\Lambda_1(\gamma)$ the supremum of the normalized first eigenvalue over all Riemannian metrics on compact orientable surfaces of genus $\gamma$:
\[
\Lambda_1(\gamma)=
\sup\big\{\bar{\lambda}_1(ds^2)\mid
ds^2 \in \mathcal{M}\text{etrics}(\gamma)
\big\}.
\]

The following theorem was previously known in some special cases, but its general form has only recently been established.

\begin{theorem}[Karpukhin, Petrides, Stern \cite{petrides2,kps}]
For every genus $\gamma$, there exists a compact Riemannian surface $(\Sigma,ds^2)$, possibly with finitely many branch points, such that
\[
\bar{\lambda}_1(ds^2)=\Lambda_1(\gamma).
\]
Moreover, the maximizing metric admits a branched minimal immersion
\[
f:(\Sigma,ds^2)\longrightarrow S^m(1)
\]
by first eigenfunctions.
\end{theorem}

The characterization of maximizing metrics in terms of minimal immersions by first eigenfunctions goes back to Nadirashvili \cite{nadir}.

\vspace{.2cm}

Equivariant versions of Petrides' existence theorem \cite{petrides} were first established by Karpukhin, Kusner, McGrath, and Stern \cite[Theorem 3.12]{kusner}, 
where the authors construct a family of news embedded compact symmetric minimal surfaces in the unit 3-sphere \(S^3(1)\).
The precise result needed in this paper is given by Vinokurov \cite[Corollary 1.4]{vino}. It applies to finite groups of holomorphic automorphisms $G\leq\operatorname{Aut}(\Sigma)$ whose orbits contain at least two points, $\mathfrak{o}(G)\geq 2$.

\begin{theorem}[\cite{kusner}, \cite{vino}]
\label{3.12}
Let $\Sigma$ be a compact Riemann surface and let $G\leq\operatorname{Aut}(\Sigma)$ be a finite group with $\mathfrak{o}(G)\geq 2$. Then there exists a $G$-equivariant harmonic map
\[
f:\Sigma\longrightarrow S^n(1)\subset\mathbb{R}^{n+1}
\]
for some integer $n\geq 1$, such that the associated conformal metric $ds_f^2$ maximizes the normalized first eigenvalue among all $G$-invariant conformal metrics on $\Sigma$:
\[
\bar{\lambda}_1(ds_f^2)
=
\sup\big\{
\bar{\lambda}_1(dt^2)\mid
dt^2 \text{ is a }G\text{-invariant conformal metric on }\Sigma
\big\}.
\]
Moreover, the coordinate functions of $f$ are first eigenfunctions of the Laplacian for $ds_f^2$.
\end{theorem}

\vspace{.2cm}

We are interested in the case where the group $G$ has large orbits. Under this hypothesis, a stronger version 
of Theorem \ref{3.12} follows by a rather direct proof linked to minimal surfaces theory.

\vspace{.2cm}

\begin{theorem} 
\label{teor11}
Let $\Sigma$ a compact Riemann surface of genus $g\geq 3$ 
and $G\leq Aut(\Sigma)$  be an holomorphic  automorphisms group whose orbits satisfy $\mathfrak{o}(G)>4(g-1)$. 
  Then,  there exists  a full, $G$-equivariant, unbranched 
  conformal minimal immersion 
  by first eigenfunctions, 
  \[
   f \hspace{-.1cm}:(\Sigma,ds^2)\longrightarrow S^n(1) \subset \R^{n+1}, n\geq 3.
  \]
Furthermore, the induced metric $ds^{2}$ is the unique one that maximizes the functional $\bar{\lambda }_{1}$ within the conformal class,
    \[
    \bar{\lambda}_1(ds^2) = 2Area(ds^2)= \Lambda_1(\Sigma).
    \]
\end{theorem}

\begin{proof}
From Theorem \ref{3.12}, there is a full $G$-equivariant harmonic map by first eigenfunctions,
\[
f:\Sigma \longrightarrow S^n(1).
\] 
Let's denote by $ds^2$ the conformal metric associated with $f$, i.e. $ds^2= ds^2_f$.  
Since the orbits of $G$ have more than $4(\gamma-1)$ points, it follows from  Lemma \ref{hopf}  that $f$ is a conformal harmonic map, i.e. $f$ is a branched isometric minimal surface by first eigenfunctions of $(\Sigma, ds^2)$ in the unit sphere $S^n(1)$, $n\geq 2$. 
 
From (\ref{liyau1}), the area of any conformal minimal surface bounds the normalized first eigenvalue within its conformal class, whether the metrics are $G$-invariant or not. Consequently we obtain the following chain of inequalities:
\[
\bar{\lambda}_1(ds^2) \leq \Lambda_1(\Sigma) \leq 2 Area(f) = \bar{\lambda}_1(ds^2),
\] 
and we get that $f$ is a branched conformal minimal immersion by first eigenfunctions.

To prove that the immersion $f$ is unbranched, using Theorem \ref{unbranched}, we bound the area of the immersion  and the length of the $G$-orbits separately. 

By Yang and Yau \cite{yy}
\[
2 Area(f) =\Lambda_1(\Sigma) \leq 8\pi \big{[} \frac{\gamma+3}{2} \big{]} \leq 4\pi(\gamma+3).
\]
On the other hand, Table (\ref{tabla1}) yields $\mathfrak{o}(G) \geq 6(\gamma-1)$, which implies
 \[
4\pi \big(  \mathfrak{o}(f)- 2(\gamma-1) \big) \geq  16\pi(\gamma -1).
 \]
 Since  $16\pi (\gamma-1 ) > 4\pi (\gamma +3)$ holds for all $\gamma \geq 3$, we conclude by Theorem   \ref{unbranched} that $f$ is unbranched.
 
 \vspace{.1cm}
 
If $n=2$, then $f: \Sigma \longrightarrow \mathbb{S}^2(1)$  is an unbranched covering map , which is a contradiction since the genus  $\gamma \geq 2$. Hence, $n\geq 3$ and the uniqueness of the metric $ds^2$ follows from Theorem \ref{mrkarpu}.
\end{proof}

The following two results reformulate Theorem \ref{teor11} in terms of the geometry of minimal immersions and their isometry groups.

\begin{theorem}
\label{family}
Let $\Sigma$ be a closed Riemann surface of genus $\gamma\geq 3$, and let $G$ be a group of holomorphic automorphisms of $\Sigma$ with signature $(2,3,7)$, $(2,3,8)$, $(2,4,5)$, or $(3,3,4)$. Then the following statements hold:

\medskip

{\rm i)} $\Sigma$ admits a $G$-equivariant, unbranched conformal minimal immersion 
\[
f:\Sigma\longrightarrow S^n(1),
\]
for some $n\geq 3$.

\medskip

{\rm ii)} The induced metric $ds^2$ is the unique metric in its conformal class that admits a minimal immersion into a unit sphere by first eigenfunctions. Moreover, it is the unique conformal metric attaining $\Lambda_1(\Sigma)$.

\medskip

{\rm iii)} For signatures $(2,3,7)$ and $(2,3,8)$, the quartic Hopf differential
\[
\Omega_2(f)=\langle f_{zz},f_{zz}\rangle\,dz^4
\]
vanishes identically.
\end{theorem}
\begin{proof}
$i)$ From Table~\ref{tabla1}, we have that the orbits for the four signatures satisfy $o(G) > 4(\gamma-1)$.
Therefore, from Theorem \ref{teor11}  we get that there is an $G$-equivariant full conformal unbranched minimal surface
$
f:\Sigma\longrightarrow S^n(1),  \hspace{.3cm}n\geq 3.
$
 Punto (ii) ya esta demostrado. Nos queda (iii)

\vspace{.1cm}

$iii)$ In all four cases, the Hopf differential $\Omega_2(f)$ is holomorphic, and it is either identically zero or has $8(\gamma -1)$ zeros (Lemma \ref{hopf}).

For Hurwitz surfaces, of signature $(2,3,7)$, we have $o(G) = 12(\gamma-1)$ and we conclude that the  Hopf differential vanishes, 
$\Omega_2(f)\equiv 0$. 

For surfaces of signature $(2,3,8)$ the lengths of the different types of orbits on the surface are
\[
\frac{48}{8}(\gamma-1) = 6(\gamma-1), \hspace{.4cm}
\frac{48}{3}(\gamma-1) = 16(\gamma-1), \hspace{.4cm}
\frac{48}{2}(\gamma-1) = 24(\gamma-1) \hspace{.2cm}{\rm and}\hspace{.2cm}
48(\gamma-1).
\] 
The second Hopf differential  has $8(\gamma-1)$ zeros;  since none of the lengths above is a divisor of $8(\gamma-1)$, it follows that $\Omega_2(f)\equiv 0$.
\end{proof}

{\bf Proof of Theorem \ref{teorema1}.}
When the group has order $84(\gamma-1)$, $48(\gamma-1)$ or $40(\gamma-1)$, the corresponding signatures are necessarily
 $(2,3,7)$, $(2,3,8)$, and $(2,4,5)$, respectively. For groups of order $24(\gamma-1)$, however, several signatures are possible; see Breuer's Lemma \cite{breuer}.
\hspace{10.5cm}
$\square$

\begin{remark}
The Bolza surface $\mathcal{B}$ is the only example of genus $\gamma=2$ in our family (see Table \ref{family3}). Its hyperelliptic map
$\Pi:\mathcal{B}\longrightarrow \mathbb{S}^2(1)\subset\mathbb{R}^3$
is a $G_{48}$-equivariant branched covering, with branch values at the eight vertices of a cube.

Nayatani and Shoda \cite{nayatani} proved that
$\Lambda_1(\mathcal{B})=\bar{\lambda}_1(ds^2_\Pi)=16\pi$
and that the coordinate functions of $\Pi$ are first eigenfunctions of the induced metric.

Moreover, this extremal metric is not unique: metrics induced by branched coverings obtained from small deformations of $\Pi$ have the same property.

While Theorem \ref{mrkarpu} ensures uniqueness for $\gamma\geq 3$, the Bolza surface provides an example of nonuniqueness in genus $2$.
\end{remark}

\vspace{.2cm}

\section{The Klein Quartic}

For any closed Riemann surface of genus $3$, the order of its holomorphic automorphism group is at most $168$, with equality attained only by the {\it Klein quartic}, the smooth plane quartic
\[
\mathcal{K}=\big\{[x,y,z]\in\mathbb{CP}^2\mid x^3y+y^3z+z^3x=0\big\}.
\]
Thus, $\mathcal{K}$ is the unique Hurwitz surface of genus $3$. In particular, any Riemannian metric $ds^2$ on a closed surface $\Sigma$ of genus $3$ whose group of orientation-preserving isometries has order $168$ belongs to the conformal class of $\mathcal{K}$.

\vspace{.1cm}

 The group   $G_{168}= Aut({\mathcal K}) \subset PGL(3,\C)$ is simple  and is generated by the following projectivities of order $7$, $3$ and $2$:

\[
\phi_7 = 
\left( 
\begin{array}{ccc}
\zeta^4 & 0&0 \\
0& \zeta^2 & 0 \\
0&0& \zeta
\end{array}
\right), \hspace{.5cm}
\phi_3 = 
\left(
\begin{array}{ccc}
0 & 0 &1 \\
1 & 0 & 0 \\
0 & 1 & 0
\end{array}
\right),
 \hspace{.5cm}
\phi_2 = \alpha
\left(
\begin{array}{ccc}
\zeta - \zeta^6 & \xi^2 - \zeta^5 & \zeta^4 - \zeta^3  \\
\zeta^2 - \zeta^5 & \zeta^4 - \zeta^3 & \zeta - \zeta^6  \\
\zeta^4 - \zeta^3 & \zeta - \zeta^6 &\zeta^2 - \zeta^5   \\
\end{array}
\right), 
\]

 where $\zeta= e^{2\pi i/7}$ and $\alpha = -1/\sqrt{-7}$, Elkies \cite{elkies}, Karcher and Weber \cite{karcher}. 

The points  
 \[
 p_1 = \left(
 \hspace{-.03cm}
\begin{array}{ccc}
1 \\
  0\\
  0
\end{array}
\hspace{-.03cm}
\right),
\hspace{.7cm}
 p_2 = \left(
\begin{array}{ccc}
0 \\
  1\\
  0
\end{array}
\right),
\hspace{.7cm}
 p_3 = \left(
\begin{array}{ccc}
0 \\
  0\\
  1
\end{array}
\right),
 \]
are the three centers of order $7$. The transformation $\phi_3$ cyclically permutes these points,
\[
\phi_3: p_1 \mapsto p_2 \mapsto p_3 \mapsto p_1.
\]
Moreover, $\phi_3$ has three fixed points in $\mathbb{CP}^2$: two centers of order $3$, $q_1,q_2\in\mathcal K$, and a third point $q_3\notin\mathcal K$, 
given by
\[
q_1 = \left(\hspace{-.1cm}
\begin{array}{ccc}
1 \\
  e^{2\pi \hspace{.02cm} i /3}\\
  e^{4\pi \hspace{.02cm} i /3}\\
\end{array}
\hspace{-.1cm}
\right),
\hspace{.7cm}
q_2 = \left(
\hspace{-.1cm}
\begin{array}{ccc}
1 \\
  e^{4\pi \hspace{.02cm} i /3}\\
  e^{2\pi \hspace{.02cm} i /3}\\
\end{array}
\hspace{-.1cm}
\right)
\hspace{.4cm} {\rm and} \hspace{.3cm}
q_3 = \left(
\begin{array}{ccc}
1 \\
1\\
  1
\end{array}
\right).
 \]
 
 \vspace{.1cm}
 
We now state other properties of the group, \cite{elkies, karcher}, and provide its real character table:

\vspace{.1cm}

$\cdot$ {\it
The group $G_{168}$ contains exactly $48$ elements of order $7$, divided into two conjugacy classes, $7A$ and $7B$, each consisting of $24$ elements, with $\phi_7\in 7A$.}

\vspace{.1cm}

$\cdot$ {\it
The $24$ centers of order $7$ form a single orbit under the action of $G_{168}$ on $\mathcal K$. These points are precisely the inflection points of the Klein quartic.}

\vspace{.1cm}

$\cdot$ {\it
The $56$ centers of order $3$ and the $84$ centers of order $2$ each form a single orbit under the action of $G_{168}$ on $\mathcal K$.}

\vspace{.1cm}

\begin{table}[H]
\[
\begin{array}{ccrrrrrrlc}
{\rm class } &  \big| & \,  {\bf 1} & {\bf 2} & {\bf 3} & {\bf 4}& {\bf 7A}& {\bf 7B}&\bf Schur& \bf Type\\
 \hline\vspace{-.3cm} & \big|\\ 
{\rm size} & \big| &  1& 21 &56 &42&24 &24\\
 \hline\vspace{-.5cm}\\ 
 \hline\vspace{-.3cm}  & \big|\\ 
 \rho_1 & \big| & 1 & 1 & 1 & 1   &1 &1 & \bf +1 &{\rm trivial}\\
 \hline\vspace{-.3cm}  & \big|\\ 
\rho_2 \oplus \rho_3 & \big|&  6& -2 &0 &2& -1 &-1 & \bf +1 &{\rm complex} 
\\
\hline\vspace{-.3cm} & \big|\\ 
\rho_4  & \big|& 6&2&0&0& -1 & -1 & \bf +1 &{\rm real} 
\\
 \hline\vspace{-.3cm} & \big|\\ 
 \rho_5 & \big| &  7& -1 &1 &-1&0 & 0  &\bf +1 &{\rm orthogonal} 
 \\
 \hline\vspace{-.3cm}  & \big|\\ 
\rho_6  & \big|&  8& 0 &-1 &0&- 1 & 1 &\bf +1 & {\rm orthogonal}
\\
\end{array}
\]

\centering
\vspace{.1cm}
\caption{Character table of the irreducible real representations of $G_{168}=GL(3,2)$. Complex character tables of finite groups of small order can be found in \cite{tim}. To obtain the characters of irreducible representations over $\mathbb R$, one must add the complex conjugate characters, such as $\rho_2$ and $\rho_3$, marked {\it complex} in the table. The remaining characters can be used directly, except for those marked {\it symplectic}, which must be doubled. The latter case does not occur for $G_{168}$.}
\label{tabla168}
\end{table}
By direct computation we get to the following fact.
\begin{lemma}
\label{lemmaklein}
The angles $\theta_i =  \frac{2\pi}{7}k_i$,  $k_i  \in \Z_7$, $i=1,2,3$,
verify 
$
2(\cos \theta_1  + \cos \theta_2  + \cos  \theta_3)  \in \Z.
$
 if and only if 

\begin{itemize}

\item[i)]
$k_1=k_2=k_3=0$ \ and  \ $ 2(\cos \theta_1  + \cos \theta_2  + \cos  \theta_3) =6$, or (up to a permutation) 


\item[ii)] $k_1=1,\ k_2=2,\  k_3=4$ \ and  \ 
$  2(\cos \theta_1  + \cos \theta_2  + \cos  \theta_3 ) =-1$, or 


\item[iii)] $k_1=6,\ k_2=5,\ k_3=3$ \ and  \ 
$  2(\cos \theta_1  + \cos \theta_2  + \cos  \theta_3 ) =-1.$
\end{itemize}
\end{lemma}

\vspace{.2cm}

In our final result, stated in the introduction, we describe the geometric properties of the minimal immersion of the Klein quartic into a sphere by first eigenfunctions.

\begin{proof}[\bf{Proof of Theorem \ref{thm:mi_teorema}}]
By Theorem \ref{family}, there exists a $G_{168}$-equivariant conformal minimal immersion $f$ of the Klein quartic into a sphere by first eigenfunctions. Moreover, if $ds^2$ is the induced metric, then
\[
2 \operatorname{Area}(f)=\bar{\lambda}_1(ds^2)=\Lambda_1(\mathcal K).
\]
Lower and upper bounds for the area of the immersion stated in {\bf iii}) are provided by the value of 
$\bar{\lambda }_{1}(ds^2_{-1})$ for  the hyperbolic metric, as estimated by Cook \cite{cook}, and by the upper bound
$\bar{\lambda}_1(dt^2) \leq 16(4-\sqrt{7})\pi$ for any metric on a compact surface of  genus \(\gamma = 3\), by Ros \cite{ros2}.

%

\vspace{.2cm}

Since the first eigenspace $V_1(ds^2)$ decomposes as a direct sum of $G_{168}$-invariant subspaces corresponding to irreducible representations, Table \ref{tabla168} yields
\[
V_1(ds^2)=V_1(\rho_1)\oplus V_1(\rho_2\oplus \rho_3)\oplus V_1(\rho_4)\oplus V_1(\rho_5)\oplus V_1(\rho_6).
\]
For the trivial representation, we have $V_1(\rho_1)=0$; see Corollary \ref{chi1}.

\vspace{.1cm}

Since the multiplicity $m_1=m_1(ds^2)$ is at most $9$ \cite{sevennec}, we conclude that $m_1\in\{6,7,8\}$. Consequently, $V_1(ds^2)$ must be one of the following spaces:
\[
V_1(\rho_2\oplus\rho_3),\qquad
V_1(\rho_4),\qquad
V_1(\rho_5),\qquad
V_1(\rho_6).
\]

\vspace{.1cm}

We shall show that the correct choice is $V_1(\rho_6)$, where $\rho_6$ denotes the irreducible representation of $G_{168}$ of dimension $8$.

\vspace{.2cm}

{\bf Claim 1.} {\it There are no $\rho$-equivariant minimal immersions $f:{\mathcal K}\longrightarrow S^5(1)\subset \R^6$ with representation 
$\rho$ equals to $\rho_2\oplus\rho_3$ or $\rho_4$.}

\vspace{.05cm}

In fact, suppose that there exists such an $f$ for one of the irreducible $6$-dimensional representations above  and
let $A_7\in O(6)$ be the orthogonal transformation associated to $\phi_7$,
\[
f\big( \phi_7(p)\big) = A_7f(p), \hspace{.3cm}\forall \hspace{.1cm} p\in {\mathcal K}.
\]
Since the character of the representation $\rho$ in the last columns 
is equal to $-1$, Table \ref{tabla168},  
from Lemma \ref{lemmaklein},  we get that  
taking suitable coordinates, 
the matrix  has the form
\begin{equation}
\setlength\arraycolsep{0pt}
\setlength\arraycolsep{0pt}
A_7 = \begin{pmatrix}
\boxed{~Rot_71} &       &        &    &   \\
           & \boxed{Rot_72} &        &      & \\
             &           & \boxed{Rot_73} &        &       \\
\end{pmatrix}   
\in O(6),
\label{r4}
\end{equation}

where the three boxes of the matriz $A_7$ correspond to planar rotations of noncero angles $\theta_1$, $\theta_2$ and $\theta_3$.
So, at the $7$-centers we  have $f(p_i)= f\circ \phi_7(p_i)= A_7 \hspace{.05cm} f(p_i)$ and, as $A_7$ has no real eigenvalues, we conclude that
$f(p_1)=f(p_2)=f(p_3)=0$ which contradicts the fact that the $f$ takes its values in the sphere. 

\vspace{.2cm}

{\bf Claim  2.} {\it There is not a $\rho_5$-equivariant conformal minimal immersion $ f : {\mathcal K}\longrightarrow S^6(1)\subset \R^7$ by first eigenfunctions.}

\vspace{.05cm}

  Assume there exists such and $f$.  In this case, 
 the order seven transformations of the representation correspond to the matrix
\begin{equation}
\setlength\arraycolsep{0pt}
\setlength\arraycolsep{0pt}
B_7 = \begin{pmatrix}
~1 ~ &       &        &    &   \\
&\boxed{Rot_71} &       &        &       \\
 &           & \boxed{Rot_72} &        &      & \\
&              &           & \boxed{Rot_73} &        &       \\
\end{pmatrix}  \in O(7)
\label{r5} 
\end{equation}
with the scalar $1$ at the first position plus three nontrivial rotational boxes of order 7, as above. 

\vspace{.1cm}

Let $L$ be the vector line of fixed vectors of $B_7$.
The argument used in claim 1 tells us that 
the vectors $f(p_1), \ f(p_2), \ f(p_3)$ belong to $S^6(1)\cap L$. As $\phi_3$ permutes the points $p_i$ we get  that 
$f(p_1)=f(p_2)=f(p_3)$.
 Then, it follows from the properties of conformal area of minimal surfaces that $Area(f)\geq 12\pi$. This is based on the fact that if a minimal immersion 
 has a triple point in its image, its area must exceed the area of three equatorial spheres (Li and Yau \cite{liyau}).  
 This contradicts item $iii)$ of the theorem, and we conclude the proof of {\bf i}): 
 
 \vspace{.1cm}
 
{\it The conformal $G_{168}$-equivariant  minimal immersion by first eigenfunctions
$
 f : {\mathcal K}\longrightarrow S^7(1)
 $ \hspace{.5cm} \linebreak
\textcolor{white}{.} \hspace{.05cm}  is  $\rho_6$-equivariant. }
 
\vspace{.1cm}

 The rest of the statements $ii)$ and $iv)$ of the theorem follow from Theorem \ref{family}.
 
 \vspace{.15cm}
 
Now we prove {\bf v}). Let \(h: \K \rightarrow S^n(1)\) be another 
$G_{168}$-equivariant conformal minimal immersion by first eigenfunctions. By Theorem \ref{mrkarpu}, the induced metric is, again, \(ds^{2}\) and from item {\bf i}) of the  theorem, que get that $n=7$ and $h$ is, up to an orthogonal transformation, $\rho_6$-equivariant, too.

  Since the linear coordinates of both immersions
\(  f,h: (\K,ds^2)\longrightarrow S^7(1),
  \) 
   lie in \(V_1(ds^2)\), 
we obtain an $8 \times 8$ matrix \(A\) such that \(h = Af\).

Equivariance implies that $A\rho_6(\phi)=\rho_6(\phi)A$ for every $\phi\in G_{168}$. Since the Frobenius–Schur indicator of $\rho_6$ is $+1$ (see Table \ref{tabla168}), the representation is absolutely irreducible. Hence, by Schur's lemma, $A=cI$ for some $c\in\mathbb R$. Since both immersions take values in the unit sphere, we have $c^2=1$, and therefore $h=\pm f$.   
\end{proof}

\begin{remark}
For hyperbolic metrics on surfaces of genus $3$, the multiplicity of the first eigenvalue is at most $8$, with equality attained by the Klein quartic \cite{maximefortier}. This surface is also conjectured to maximize the first eigenvalue among all hyperbolic surfaces of genus $3$ \cite{cook}.
\end{remark}

\vspace{.2cm}

To conclude, we point out two natural 
open problems concerning the first eigenvalue on surfaces of genus \(3\).
Let $\mathcal{M}\mathrm{etrics}(3)$ denote the space of Riemannian metrics on a closed surface of genus $3$, and consider the first eigenvalue functional
 \(
\bar{\lambda}_1:\mathcal{M}\mathrm{etrics}(3)\longrightarrow\mathbb{R}.
\)

\vspace{.06cm}

\hspace{.3cm} a) {\it Determine the Riemannian metric that maximizes the functional $\bar{\lambda}_1$.}

\vspace{.05cm}

The natural candidate is the metric described in Theorem \ref{thm:mi_teorema}.

\vspace{.1cm}

\hspace{.3cm} b) {\it Determine all critical points of $\bar{\lambda}_1$, or equivalently, classify minimal immersions by first eigenfunctions.}

\vspace{.1cm}

Among these critical points are the conformally equivariant minimal immersions associated with the Klein and Fermat quartics (see Table \ref{family3}), as well as the Lawson and Karcher--Pinkall--Sterling minimal surfaces of genus $3$ in $S^3(1)$ \cite{choe1}.

{\footnotesize

\noindent

{Antonio Ros,} { aros@ugr.es}  \vspace{-.10cm} \\
Department of Geometry and Topology and \vspace{-.10cm}\\
Institute of Mathematics (IMAG), \vspace{-.10cm}\\
University of Granada\vspace{-.10cm}\\
 18071 Granada, Spain.
 }

\end{document}